\documentclass[10pt]{amsproc}%
\usepackage[initials]{amsrefs}
\usepackage{amssymb,amsmath,amsfonts,latexsym,amsthm,geometry}
\usepackage{graphicx}
\usepackage{amsmath}
\usepackage{amsfonts}
\usepackage{amssymb}%
\providecommand{\U}[1]{\protect\rule{.1in}{.1in}}
\theoremstyle{plain}

\newtheorem{example}{Example}

\newtheorem{proposition}{Proposition}

\numberwithin{equation}{section}
\begin{document}
\author[Pellegrino]{Daniel Pellegrino}
\address{Departamento de Matem\'{a}tica, \\
\indent
Universidade Federal da Para\'{\i}ba, \\
\indent
58.051-900 - Jo\~{a}o Pessoa, Brazil.}
\email{dmpellegrino@gmail.com and pellegrino@pq.cnpq.br}
\author[Seoane]{Juan B. Seoane-Sep\'{u}lveda}
\address{Departamento de An\'{a}lisis Matem\'{a}tico,\\
\indent
Facultad de Ciencias Matem\'{a}ticas, \\
\indent
Plaza de Ciencias 3, \\
\indent
Universidad Complutense de Madrid,\\
\indent
Madrid, 28040, Spain.}
\email{jseoane@mat.ucm.es}
\title[Subpolynomial constants in the Bohnenblust--Hille inequality via interpolation]{An interpolation technique towards the subpolynomial constants in the
multilinear Bohnenblust--Hille inequality}
\date{}
\keywords{Bohnenblust--Hille inequality; Littlewood's $4/3$ inequality; Absolutely
summing operators}
\maketitle

\begin{abstract}
The multilinear Bohnenblust--Hille inequality, in a recent more general
presentation, asserts that if $q_{1},...,q_{n}\in\lbrack1,2]$ and $\frac
{1}{q_{1}}+\cdots+\frac{1}{q_{n}}=\frac{n+1}{2}$, then there is a constant
$K=K_{q_{1}...q_{n}}\geq1$ such that{\small {
\[
\left(  {\sum\limits_{i_{1}=1}^{\infty}}\left(  {\sum\limits_{i_{2}=1}%
^{\infty}}\left(  ...\left(  {\sum\limits_{i_{n-1}=1}^{\infty}}\left(
{\sum\limits_{i_{n}=1}^{\infty}}\left\vert A\left(  e_{i_{1}},...,e_{i_{n}%
}\right)  \right\vert ^{q_{n}}\right)  ^{\frac{q_{n-1}}{q_{n}}}\right)
^{\frac{q_{n-2}}{q_{n-1}}}\cdots\right)  ^{\frac{q_{2}}{q_{3}}}\right)
^{\frac{q_{1}}{q_{2}}}\right)  ^{\frac{1}{q_{1}}}\leq K\left\Vert
A\right\Vert
\]
}}for all continuous $n$-linear forms $A:c_{0}\times\cdots\times
c_{0}\rightarrow\mathbb{K};$ the original construction provides constants
$K_{q_{1}...q_{n}}=\left(  \sqrt{2}\right)  ^{n-1}$ for real scalars and
$K_{q_{1}...q_{n}}=\left(  2/\sqrt{\pi}\right)  ^{n-1}$ for complex scalars.
In this note we present a new interpolative approach which provides quite
better constants. Our procedure, when restricted to the original
Bohnenblust--Hille inequality, gives a very simple and self-contained
interpolative proof of the Bohnenblust--Hille inequality with the best known
constants (with subpolynomial growth), which avoids the technical issues of
the original proof. This seems to be unexpectedly surprising since the known
interpolative approaches to the Bohnenblust--Hille inequality only provide
constants having exponential growth.

\end{abstract}


\section{Introduction}

Recall that that the multilinear Bohnenblust-Hille inequality for
$\mathbb{K}=\mathbb{R}$ or $\mathbb{C}$ (see \cite{bh}) asserts that for every
positive integer $n\geq1$ there exist positive scalars $C_{n}\geq1$ such that
\begin{equation}
\left(  \sum\limits_{i_{1},\ldots,i_{n}=1}^{\infty}\left\vert A(e_{i_{^{1}}%
},\ldots,e_{i_{n}})\right\vert ^{\frac{2n}{n+1}}\right)  ^{\frac{n+1}{2n}}\leq
C_{n}\left\Vert A\right\Vert \label{hypp}%
\end{equation}
for all $n$-linear forms $A:c_{0}\times\cdots\times c_{0}\rightarrow
\mathbb{K}$, where $e_{i}$ are the canonical vectors of $c_{0}.$ A very recent
generalization of the (multilinear) Bohnenblust--Hille inequality was
presented in \cite{ba}, highlighting the importance of interpolation arguments
in this framework. Namely, in \cite{ba}, with a new interpolative approach, it
was proved that, if $n\geq1$ and $q_{1},...,q_{n}\in\lbrack1,2],$ then the
following assertions are equivalent:

\begin{itemize}
\item[(1)] There is a constant $K_{q_{1}...q_{n}}\geq1$ such that{\small {
\begin{equation}
\left(  {\textstyle\sum\limits_{i_{1}=1}^{\infty}}\left(  {\textstyle\sum
\limits_{i_{2}=1}^{\infty}}\left(  ...\left(  {\textstyle\sum\limits_{i_{n-1}%
=1}^{\infty}}\left(  {\textstyle\sum\limits_{i_{n}=1}^{\infty}}\left\vert
A\left(  e_{i_{1}},...,e_{i_{n}}\right)  \right\vert ^{q_{n}}\right)
^{\frac{q_{n-1}}{q_{n}}}\right)  ^{\frac{q_{n-2}}{q_{n-1}}}\cdots\right)
^{\frac{q_{2}}{q_{3}}}\right)  ^{\frac{q_{1}}{q_{2}}}\right)  ^{\frac{1}%
{q_{1}}}\leq K_{q_{1}\ldots q_{n}}\left\Vert A\right\Vert \label{75}%
\end{equation}
}}for all continuous $n$-linear forms $A:c_{0}\times\cdots\times
c_{0}\rightarrow\mathbb{K}$.

\item[(2)] $\frac{1}{q_{1}}+\cdots+\frac{1}{q_{n}}\leq\frac{n+1}{2}.$
\end{itemize}

\noindent In the case
\[
q_{1}=\cdots=q_{n}=\frac{2n}{n+1}%
\]
we recover the classical Bohnenblust--Hille inequality. However, the constants
$K_{q_{1}\ldots q_{n}}$ in the extremal case, i.e., $\frac{1}{q_{1}}%
+\cdots+\frac{1}{q_{n}}=\frac{n+1}{2},$ arisen from this new approach have an
exponential growth; more precisely, $K_{q_{1}...q_{n}}=\left(  \sqrt
{2}\right)  ^{n-1}$ for real scalars and $K_{q_{1}...q_{n}}=\left(
2/\sqrt{\pi}\right)  ^{n-1}$ for complex scalars (see \cite[Remark 5.1]{ba}),
and this may be a little bit disappointing at a first glance, having in mind
that the optimal constants of the multilinear Bohnenblust--Hille inequality
have a subpolynomial growth (see \cite{jfalimite}).

In this note we present a new interpolative argument which generates quite
better constants for the constants $K_{q_{1}...q_{n}}$. As an illustration of
the effectiveness of this method we show, in details, in the Section 3, that
we recover the best known constants of the Bohnenblust--Hille inequalities
(that can be found in \cite{jfa2} and \cite{computer} and relies on results
from \cite{def}) in an elementary form, avoiding the technicalities from
\cite{def} (for instance, the technical variant of an inequality due to Blei
(see \cite[Lemma 3.1]{def})).

\section{The new interpolative approach}

From now on, for positive integers $n\geq1,$ the symbols $C_{n}$ and
$\widetilde{C_{n}}$ denote the optimal constants of the Bohnenblust--Hille
inequalities for real and complex scalars, respectively. Also, for
$p\in\left[  1,2\right]  ,$ the symbols $A_{p}$ and $\widetilde{A_{p}}$
denote the optimal constants of the Khinchin inequality with Rademacher (real
case) and Steinhaus (complex case) variables, respectively. Finally,
$\left(  q_{1},...,q_{n}\right)  $ shall be called a Bohnenblust--Hille exponent if we have
$q_{1},...,q_{n}\in\left[  1,2\right]  $ and
\[
\frac{1}{q_{1}}+\cdots+\frac{1}{q_{n}}=\frac{n+1}{2}.
\]

\bigskip In \cite{ba} it was shown that each Bohnenblust--Hille exponent
$\left(  q_{1},...,q_{n}\right)  $ is generated by interpolating the $n$
Bohnenblust--Hille exponents $\left(  2,2,...,2,1\right)  ,\left(
2,2,...,2,1,2\right)  ,...,\left(  1,2,2,...,2\right)  .$ This construction
provides, at the end, estimates $\left(  \sqrt{2}\right)  ^{n-1}$ and $\left(
\frac{2}{\sqrt{\pi}}\right)  ^{n-1}$ for $K_{q_{1}...q_{n}}$ (\cite[Remark
5.1]{ba}), for real and complex scalars, respectively. However, and as mentioned
before, when we restrict our attention to the classical Bohnenblust--Hille
exponents, since the optimal constants are subpolynomial, the above estimates
are \textit{quite bad}. The following simple result shows that the estimates
above are also far from being good, even for general Bohnenblust--Hille exponents. This result also gives us a family of Bohnenblust--Hille exponents with
small constants (in some sense) which shall be used to generate, by interpolation, other Bohnenblust--Hille exponents with \emph{small} constants (see Example \ref{099a}).

\begin{proposition}
\label{099}The optimal constants associated to the Bohnenblust--Hille
exponents $\left(  q_{1},...,q_{n}\right)  $ with $q_{j}=\frac{2k}{k+1}$ for
$k$ indexes $j$ and $q_{j}=2$ for $n-k$ indexes $j$ are smaller than%
\begin{equation}
C_{k}\left(  A_{\frac{2k}{k+1}}\right)  ^{-\left(  n-k\right)  }\text{ and
}\widetilde{C_{k}}\left(  \widetilde{A_{\frac{2k}{k+1}}}\right)  ^{-\left(
n-k\right)  }\label{q2}%
\end{equation}
for real and complex scalars, respectively.
\end{proposition}

\begin{proof}
\bigskip A simple adaptation of \cite[Prop. 3.1]{ba} tells us that we can
consider $q_{1}=\cdots=q_{k}=\frac{2k}{k+1}$ in (\ref{75})$.$ The result is
obtained by using the multiple Khinchin inequality for Steinhaus variables for
complex scalars (see \cite[Theorem 2.2]{jfa2}) and the multiple Khinchin
inequality for Rademacher functions for real scalars (see \cite[Theorem
1.3]{jmaa}).
\end{proof}

\bigskip A simple calculus shows us that the constants in \eqref{q2} are smaller
than $\left(  \sqrt{2}\right)  ^{n-1}$ and $\left(  \frac{2}{\sqrt{\pi}%
}\right)  ^{n-1}.$ The next example shows that the same happens in other situations:

\begin{example}
\label{099a}
Consider the Bohnenblust--Hille exponent $\left(  q,q,q,r,r,s,t,u\right)  $
with $q,r,s,t,u$ being pairwise distinct and $s>t>u.$ Since%
\[
\frac{3}{q}+\frac{2}{r}+\frac{1}{s}+\frac{1}{t}+\frac{1}{u}=\frac{9}{2},
\]
a simple computation shows that%
\[
\left(  q,r,s,t,u\right)  \in\left[  \frac{3}{2},2\right]  \times\left[
\frac{4}{3},2\right]  \times\left[  \frac{3}{2},2\right]  \times\left[
\frac{4}{3},2\right)  \times\left[  1,2\right)  .
\]
Then we interpolate%
\[
\left(  \frac{3}{3},\frac{3}{2},\frac{3}{2},2,...,2\right)  ,\left(
2,2,2,\frac{4}{3},\frac{4}{3},2,2,2\right)  ,\left(  2,...,2,\frac{3}{2}%
,\frac{3}{2},\frac{3}{2}\right)  ,\left(  2,...,2,\frac{4}{3},\frac{4}%
{3}\right)  ,\left(  2,...,2,1\right)
\]
with%
\[
\left(  \theta_{j}\right)  _{j=1}^{4}=\left(  2\left(  3-\frac{2}{r}-\frac
{1}{s}-\frac{1}{t}-\frac{1}{u}\right)  ,-2\left(  \frac{-2+r}{r}\right)
,-3\left(  \frac{-2+s}{s}\right)  ,4\left(  \frac{s-t}{st}\right)  ,2\left(
\frac{t-u}{tu}\right)  \right)  ,
\]
to conclude that (for real scalars) the constant $K_{qqqrrstu}$ is dominated
by%
\[
\left(  C_{3}\left(  A_{\frac{3}{2}}\right)  ^{-5}\right)  ^{\theta_{1}%
+\theta_{3}}\left(  C_{2}\left(  A_{\frac{4}{3}}\right)  ^{-6}\right)
^{\theta_{2}+\theta_{4}}\left(  \sqrt{2}^{7}\right)  ^{\theta_{5}}.
\]
Note that this estimate is quite better than $\left(  \sqrt{2}\right)  ^{7}.$
\end{example}

\bigskip The idea is that there are several different interpolative approaches
that generate a given Bohenenblust--Hille exponent $\left(  q_{1}%
,...,q_{n}\right)  .$ If $\left(  q_{1},...,q_{n}\right)  $ is generated by
the Bohnenblust--Hille exponents $$\alpha_{1}=\left(  \alpha_{11}%
,...,\alpha_{1n}\right)  ,...,\alpha_{j}=\left(  \alpha_{j1},...,\alpha
_{jn}\right)  ,$$ let us denote $\left(  q_{1},...,q_{n}\right)  \in<\alpha
_{1},...,\alpha_{j}>$ and represent the constant derived from the
interpolation of $\alpha_{1},...,\alpha_{n}$ by $C_{\alpha_{1},...,\alpha_{n}%
}.$ Then, clearly,%
\[
K_{q_{1}...,q_{n}}\leq\inf\left\{  C_{\alpha_{1},...,\alpha_{n}}:\left(
q_{1},...,q_{n}\right)  \in<\alpha_{1},...,\alpha_{j}>\text{ and }%
j\in\mathbb{N}\right\}  .
\]

As the previous example suggests, it seems that better constants are derived
from interpolation with less steps. Also, it seems that sometimes a different
choice of Bohnenblust--Hille exponents, although using the same number of
interpolative steps, provides better constants. For instance, $\left(
q_{1},q_{2},q_{3}\right)  $ with $q_{1}>q_{2}>q_{3}$ can be generated by%
\[
\left(  2,2,1\right)  ,\left(  2,\frac{4}{3},\frac{4}{3}\right)  ,\left(
\frac{3}{2},\frac{3}{2},\frac{3}{2}\right)
\]
with adequate interpolation weights $\theta_{1},\theta_{2}$ and $\theta_{3}$,
respectively. The resulting constant for real scalars (and using Proposition
\ref{099}) is%
\[
2^{\theta_{1}}\left(  C_{2}A_{\frac{4}{3}}\right)  ^{-\theta_{2}}\left(
C_{3}\right)  ^{\theta_{3}},
\]
which, as it can be easily checked, is smaller than$\left(  \sqrt{2}\right)
^{2}$, which is generated by the interpolation of $\left(  2,2,1\right)
,\left(  2,1,2\right)  ,$ and $(2,2,1).$

\bigskip In the next section we apply our approach to recover, in a very
simple and quick way, the best known constants for the classical
Bohnenblust--Hille inequality.

\section{The interpolative proof of multilinear the Bohnenblust--Hille
inequality with subpolynomial constants}

The first interpolative proof of the Bohnenblust--Hille inequality is probably
due to Kaijser (\cite{Ka}, see also \cite{surv} for details); this proof,
however, gives constants with exponential growth. As mentioned before, the
same exponential growth appears in the recent interpolative proof of the
general Bohnenblust--Hille inequality from \cite{ba}. In this section we
illustrate, in details, the particular case our our approach when we deal with
the classical Bohnenblust--Hille inequality: we recover the best known
constants straightforwardly. Since the constants of the real and complex cases
are obtained via different approaches (this strange fact was recently stressed
in \cite{computer}) we will divide the proof in two different cases.

\subsection{Case of complex scalars}

The constant $\widetilde{C_{2n}}$ will be derived from the constant
$\widetilde{C_{n}}$. As mentioned in
the previous section, from the multiple Khinchin inequality for Steinhaus
variables a straightforward computation provides the constant%
\[
\widetilde{C_{n}}\left(  \widetilde{A_{\frac{2n}{n+1}}}\right)  ^{-n}%
\]
for the Bohnenblust--Hille exponents $\left(  \overset{n\text{ times}}%
{\frac{2n}{n+1},...,\frac{2n}{n+1}},\overset{n\text{ times}}{2,...,2}\right)
$ and $\left(  \overset{n\text{ times}}{2,...,2},\overset{n\text{ times}%
}{\frac{2n}{n+1},...,\frac{2n}{n+1}}\right)  .$ Now, interpolating these
exponents in the sense of \cite{ba} with $\theta_{1}=\theta_{2}=1/2$ we get%
\begin{equation}
\widetilde{C_{2n}}\leq\widetilde{C_{n}}\left(  \widetilde{A_{\frac{2n}{n+1}}%
}\right)  ^{-n}. \label{90}%
\end{equation}
For the case $2n+1$, again choosing $k_{1}=n$ and $r=2$, and using the
multiple Khinchin inequality for Steinhaus variables we obtain the constants%
\[
\widetilde{C_{n}}\left(  \widetilde{A_{\frac{2n}{n+1}}}\right)  ^{-n-1}\text{
and }\widetilde{C_{n+1}}\left(  \widetilde{A_{\frac{2\left(  n+1\right)
}{\left(  n+1\right)  +1}}}\right)  ^{-n}%
\]
for the Bohnenblust--Hille exponents $\left(  \overset{n\text{ times}}%
{\frac{2n}{n+1},...,\frac{2n}{n+1}},\overset{n+1\text{ times}}{2,...,2}%
\right)  $ and $\left(  \overset{n\text{ times}}{2,...,2},\overset{n+1\text{
times}}{\frac{2\left(  n+1\right)  }{\left(  n+1\right)  +1},...,\frac
{2\left(  n+1\right)  }{\left(  n+1\right)  +1}}\right)  $, respectively. Now
we interpolate the above Bohnenblust--Hille exponents with $\theta_{1}%
=\frac{n}{2n+1}$ and $\theta_{2}=\frac{n+1}{2n+1},$ respectively, and we get%
\begin{equation}
\widetilde{C_{2n+1}}\leq\left(  \widetilde{C_{n}}\left(  \widetilde
{A_{\frac{2n}{n+1}}}\right)  ^{-n-1}\right)  ^{\frac{n}{2n+1}}\left(
\widetilde{C_{n+1}}\left(  \widetilde{A_{\frac{2\left(  n+1\right)  }{\left(
n+1\right)  +1}}}\right)  ^{-n}\right)  ^{\frac{n+1}{2n+1}}. \label{92}%
\end{equation}
Note that the formulas (\ref{90}) and (\ref{92}) are precisely those from
\cite{jfa2} which generates the best known constants for the
Bohnenblust--Hille inequality for complex scalars.

\subsection{Case of real scalars\label{9870}}

The exactly same proof of the previous case, using now Rademacher functions
instead of Steinhaus variables and the corresponding multiple Khinchin
inequality, gives us the estimates from \cite{jmaa}. However, it was very
recently shown, in \cite{computer}, that in the case of real scalars the
estimates from \cite{jmaa} can be improved by using a somewhat chaotic
combinatorial approach. For instance, in \cite{computer} it was shown that a
better (smaller) constant for $m=26$ was obtained by combining the cases
$m=12$ and $m=14$ instead of using the case $m=13;$ it was also shown that
this also happens in several other cases. Of course, our interpolation
technique with a different choice of $k_{1}$ and $r,$ in order to compute
$C_{26}$ using $C_{12}$ and $C_{14}$, as well as any other choices of
combinations, gives us exactly the same constants from \cite{computer}.

\bigskip

\noindent\textbf{Acknowledgments}. This note was written during D.
Pellegrino's visit to Spain to participate in the $2^{\circ}$ Congreso de
J\'ovenes Investigadores - RSME 2013, held in Sevilla. He thanks CAPES for the
financial support.


\end{document}